\documentclass{gLMA2e}

\usepackage{epstopdf}
\usepackage{subfigure}

\theoremstyle{plain}
\newtheorem{theorem}{Theorem}[section]
\newtheorem{corollary}[theorem]{Corollary}
\newtheorem{lemma}[theorem]{Lemma}
\newtheorem{proposition}[theorem]{Proposition}

\theoremstyle{definition}

\newtheorem{example}{Example}

\theoremstyle{remark}
\newtheorem{remark}{Remark}

\newtheorem{observation}[theorem]{Observation} 
\newtheorem{question}[theorem]{Question}
\newtheorem{conj}[theorem]{Conjecture} 

\newcommand{\R}{\mathbf{R}}  
\newcommand{\Z}{\mathbf{Z}}  

\begin{document}


  \title{The critical exponent for generalized doubly nonnegative matrices}


\author{
\name{Xuchen Han\textsuperscript{a}$^{\ast}$$^{\dagger}$\thanks{$^\ast$Corresponding author. Email: xhan@math.ucla.edu}\thanks{$^\dagger$This author was in part supported by SIGP}, Charles R. Johnson\textsuperscript{b} and Pietro Paparella\textsuperscript{c}}
\affil{\textsuperscript{a}Department of Mathematics, University of California Los Angeles, Los Angeles, CA 90095, USA
\textsuperscript{b}Department of Mathematics, College of William and Mary, Williamsburg, VA 23187-8795, USA
\textsuperscript{c}Division of Engineering and Mathematics, University of Washington Bothell, Bothell, WA 98011-8246, USA}
}

\maketitle


\begin{abstract}
		It is known that the critical exponent (CE) for conventional, continuous powers of $n$-by-$n$ doubly nonnegative (DN) matrices is $n-2$. Here, we consider the larger class of diagonalizable, entrywise nonnegative $n$-by-$n$ matrices with nonnegative eigenvalues (GDN). We show that, again, a CE exists and are able to bound it with a low-coefficient quadratic. However, the CE is larger than in the DN case; in particular, 2 for $n=3$. There seems to be a connection with the index of primitivity, and a number of other observations are made and questions raised. It is shown that there is no CE for continuous Hadamard powers of GDN matrices, despite it also being $n-2$ for DN matrices.
\end{abstract}

 \begin{keywords} Critical exponent, nonnegative matrix, index of primitivity, continuous conventional power, generalized doubly nonnegative matrix
\end{keywords}

\begin{classcode}
Primary 15B48
\end{classcode}


  


       \date{\today}



\pagenumbering{arabic}

\section{Introduction}

An $n$-by-$n$ real symmetric matrix is called \emph{doubly nonnegative} (DN) if it is both positive semidefinite and entrywise nonnegative.
Given a doubly nonnegative matrix $A$, the continuous conventional powers of $A$ are defined using the spectral decomposition: 
if $\alpha>0$, $A = UDU^{T}$, and $D = \text{diag}(d_{11}, . . . , d_{nn})$, then $A^\alpha := UD^\alpha U^T$, 
with $D^\alpha := \text{diag}(d_{11}^\alpha, . . . , d_{nn}^\alpha)$.
The conventional \emph{critical exponent} (CE) for DN matrices is the least real number $m$ such that for any DN matrix $A$, $A^\alpha$ is also DN for all $\alpha > m$. It was shown in \cite{Johnson} that the critical exponent for DN matrices exists and is no smaller than $n-2$, and is $n-2$ for $n<6$. A low coefficient quadratic upper bound was also given in \cite{Johnson}. The authors conjectured that the critical exponent is $n-2$, and this conjecture was proven in \cite{Stanford} by applying a result from \cite{M}. There is also the concept of critical exponents of DN matrices under Hadamard powering, and interestingly enough, the critical exponent is also shown to be $n-2$ in \cite{Hadamard}. 

Here we relax the assumption that the matrix be symmetric while still insisting that the matrix is entrywise nonnegative, diagonalizable, and has nonnegative eigenvalues. We call such matrices \emph{generalized doubly nonnegative} (GDN). Because GDN matrices are diagonalizable, we have 
$A = SDS^{-1}$, where $D$ is a diagonal matrix, and we can define continuous powers similarly via $A^\alpha := SD^\alpha S^{-1}$.

We show here the critical exponent for GDN matrices also exists, and we give low-coefficient quadratic upper bounds for it. We show that the critical exponent is strictly larger than $n-2$ if $n$ is an odd integer greater than 2. We make observations about the relation between the index of primitivity of a primitive matrix and the critical exponent for GDN matrices of that size. In addition, we make the observation that the GDN critical exponent under Hadamard powering does not exist.


\section{Background}

Any diagonalizable matrix $A \in M_n(\R)$ can be decomposed as 
\[ A = SDS^{-1}\]
in which $D$ is a diagonal matrix consisting of the eigenvalues of $A$.
If $x_i$ denotes the $i\textsuperscript{th}$-column of $S$ and $y_i$ denotes the $i\textsuperscript{th}$-row of $S^{-1}$, then $A$ can be written as 
\[ A = \lambda_1x_1y_1 + ... + \lambda_nx_ny_n. \]

If all eigenvalues of $A$ are nonnegative, then for $\alpha>0 $, $A^\alpha$ is defined by
\[ A^\alpha = \lambda_1^\alpha x_1y_1+ ... + \lambda_n^\alpha x_ny_n. \]
Each entry of $A^\alpha$ has the form
\[ (A^\alpha)_{ij} = \lambda_1^\alpha (x_1y_1)_{ij}+...+ \lambda_n^\alpha (x_ny_n)_{ij}. \]
Any function of the form
\[ \phi(t) = a_1 e^{b_1 t} + ... +a_n e^{b_n t} \]
where $a_i, b_i \in \R$, is called an \emph{exponential polynomial}. In particular, if all eigenvalues of $A$ are nonnegative, then each entry of $A^\alpha$ is an
exponential polynomial in $\alpha$. The eigenvalues of the $A^\alpha$ are obvious but the non-negativity of the entries is not obvious. The following version of Descartes' rule for exponential polynomials is
well known and appears as an exercise in \cite{DR}.

\begin{lemma}\label{DR}
Let $ \phi(t) = \sum_{i=1}^{n} a_i e^{b_i t}$ be a real exponential polynomial such that each $a_i \not = 0$ and
$b_1  > b_2 > . . . > b_n$. The number of real roots of $\phi(t)$, counting multiplicity, cannot exceed the number
of sign changes in the sequence of coefficients $(a_1, a_2, . . . , a_n)$.
\end{lemma}


\section{The existence and an upper bound for the GDN critical exponent}

We follow the strategy of \cite{Johnson} to show the existence and an upper bound for CE.
Lemma \ref{DR} leads immediately to the the existence of a GDN CE.

\begin{theorem}\label{existence}
There is a function $m(n)$ such that for any $n$-by-$n$ GDN matrix $A$, $A^\alpha$ is generalized doubly
nonnegative for $\alpha \ge m(n)$.
\end{theorem}

The proof of Theorem 2.1 in \cite{Johnson} does not rely on the symmetry assumption so that essentially the same proof establishes Theorem \ref{existence}. 

\begin{proof}
Let $A$ be an $n$-by-$n$ GDN matrix. Since $A$ is entrywise nonnegative, so is $A^k$ for all positive
integers $k$. If $A^\alpha$ is entrywise nonnegative for all $\alpha \in [m,m + 1]$, where $m \in \Z$, then it follows from repeated
multiplication by $A$ that $A^\alpha$ is also entrywise nonnegative for all $\alpha \ge m$. Suppose that $A^\alpha$ has a negative entry for some
$\alpha \in [m,m + 1]$, then the exponential polynomial corresponding to that entry must have at least two
roots in the interval $[m,m+1]$ by continuity and the fact that $A^m$ and $A^{m+1}$ are both entrywise nonnegative. By Lemma \ref{DR}, the maximum number of roots each entry may possess
depends on $n$. It follows that there is a constant $m(n)$ such that $A^\alpha $ is entrywise nonnegative, and thus GDN, for all $\alpha > m(n)$. 
\end{proof}

Moreover, we may strengthen the argument in the proof of Theorem \ref{existence} to give an upper bound for the CE after developing some tools.

Let $A$ be any $n$-by-$n$ GDN matrix. Following \cite{Johnson}, corresponding to the matrix $A$, we define a matrix
$W = [w_{ij}]$ where $w_{ij}$ equals the number of sign changes in the sequence of coefficients of the exponential polynomial $(A^\alpha)_{ij}$ arranged in decreasing order of the corresponding eigenvalues. We refer
to $W$ as the \emph{sign change matrix} for $A$. By Lemma \ref{DR}, each entry $w_{ij}$ of a
sign change matrix gives an upper bound on the number of real zeros of the corresponding exponential
polynomial $(A^\alpha)_{ij}$, counting multiplicity. Note that $w_{ij} \le n-1$ because there are at most $n$ terms in the exponential polynomial $(A^\alpha)_{ij}$.

\begin{lemma}\label{2.3}

Let $A$ be an invertible GDN matrix with sign change matrix $W = [w_{ij}]$. Let $\bar{T}_{ij}= \{ \alpha > 1 : A^\alpha_{ij} < 0 \}$. Then the maximum number of connected components of $\bar{T}_{ij}$ is 
\[\begin{cases}
    \left \lfloor{\frac{w_{ij}-1}{2}}\right \rfloor   &\text{    if $w_{ij} > 0$ and $i \not = j$} \\
    	   \left \lfloor{\frac{w_{ij}}{2}}\right \rfloor &  \text{    if $w_{ij} > 0$ and $i  = j$} \\	
           0 & \text{          if $w_{ij} = 0$}
\end{cases}.\]
 \end{lemma}

\begin{proof}

By Lemma \ref{DR} the maximum number of real roots of the exponential polynomial $A^\alpha_{ij}$ is given
by $w_{ij}$. Since $A$ is invertible, the exponential polynomials defining the entries of $A^\alpha$ when $\alpha > 0$ still
agree with $A^\alpha$ at $\alpha = 0$. Since $A^0$ is the identity matrix, the exponential polynomial $(A^\alpha)_{ij}$ has at most
$w_{ij} - 1$ roots in the interval $[1,\infty)$ when $i \not = j$.
Each of the connected components of $\bar{T}_{ij}$ is bounded because $A^k$ is nonnegative for all positive
integers $k$. The endpoints of these components are roots of the exponential polynomial $(A^\alpha)_{ij}$. If two
adjacent connected components of $\bar{T}_{ij}$ share an endpoint, that endpoint must be a root of degree at
least two. Counting multiplicity, the number of real roots of $(A^\alpha)_{ij}$ with $\alpha \ge 1$ must therefore be at
least double the number of connected components of $\bar{T}_{ij}$.
If $w_{ij}$ is zero, then the exponential polynomial $(A^\alpha)_{ij}$ has all positive coefficients, so $\bar{T}_{ij}$ is empty.
And if $w_{ii}$ is not zero, the corresponding exponential polynomial has at most $w_{ij}$ roots counting multiplicity and $0$ is not one of them. So there are at most  $\left \lfloor{\frac{w_{ij}}{2}}\right \rfloor$ number of connected components.
\end{proof}

From now on, we will denote the GDN critical exponent of $n$-by-$n$ matrices by $CE_n$. We now provide an upper bound for $CE_n$. 

\begin{theorem}\label{upper}
We have
\[ CE_n \le \begin{cases} 
      \frac{n^2 -3n+4}{2} & \text{n is odd} \\
      \frac{n^2 - 2n}{2} & \text{n is even} \\
   \end{cases}.
\]
\end{theorem}

\begin{observation}\label{ob}
For $j\in \{1,...,n\}$, let $\bar{T}_{j} = \{ \alpha > 1 : {A^\alpha}_{ij} < 0, i = 1, ..., n \}$.
Note that if $\bar{T}_{j} \cap (m,m + 1) = \emptyset$, for some integer $m$, then every entry in
column $j$ of $A^\alpha$ is nonnegative for all powers $\alpha \in [m,m + 1]$. Using repeated left multiplication by $A$,
we see that column $j$ of $A^\alpha$ must be nonnegative for all $\alpha \ge m$.
\end{observation}

\begin{proof}
Let $k(n)$ be the proposed upper bound. Assume $A$ is an invertible GDN matrix.

If $n$ is odd, then $\bar{T}_{j}$ can have at most 
\begin{align*}
(n-1)\left \lfloor{\frac{n-2}{2}}\right \rfloor + \left \lfloor{\frac{n-1}{2}}\right \rfloor = \frac{n^2 -3n+2}{2} 
\end{align*}
connected components in $[1, \infty)$. By Observation \ref{ob}, there has to be a connected component in $(0,1)$ as well. But this can be achieved by letting the exponential polynomial for one of the off-diagonal entries to have exactly one simple root in $(0,1)$.

If $n$ is even, then $\bar{T}_{j}$ can have at most 
\begin{align*}
(n-1)\left \lfloor{\frac{n-2}{2}}\right \rfloor + \left \lfloor{\frac{n-1}{2}}\right \rfloor = \frac{n^2 -2n}{2} 
\end{align*}
connected components in $[1, \infty)$. Again by Observation \ref{ob}, there has to be a connected component in $(0,1)$ as well. But in this case, the number of connected components in $[1, \infty)$ has to decrease by at least 1 if we insist that there is a connected component in $(0,1)$. Therefore, there are at most $\frac{n^2 -2n-2}{2}$ connected components in $[1, \infty)$. 
Finally, by Observation \ref{ob}, the connected components in $[1, \infty)$ have to lie in intervals with consecutive integers as end points, starting from $[1,2]$. Therefore, there are no such connected components in $(k(n), \infty)$

Now suppose that $A$ is singular. By continuity, $A^\alpha$ cannot have a negative entry for any $\alpha > k(n)$.
Therefore the critical exponent $CE_n \le k(n)$.
\end{proof}

\section{The GDN critical exponent and the index of primitivity}
We first note that since a DN matrix is also GDN, the critical exponent for GDN matrices is no smaller than the critical exponent for DN matrices.

We now focus on irreducible matrices and explore the relation between the GDN critical exponent and the index of primitivity. We will address reducible GDN matrices in the next section.

 If an $n$-by-$n$ GDN matrix is irreducible, then it is primitive by the Perron-Frobenius theorem. The \emph{index of primitivity} of a primitive matrix $A$ is the least positive integer $k$ such that $A^k$ is entrywise positive. It is known that the index of primitivity of a $n$-by-$n$ matrix is at most $(n-1)n^n$ (NOTE:CITE HORN,JOHNSON MATRIX ANALYSIS HERE).  We denote the maximum index of primitivity for primitive $n$-by-$n$ GDN matrices by $MIP_n$. 
By the definition of index of primitivity, there exists a GDN matrix that has at least one zero entry, say the $(i,j)$-entry, when raised to the power $MIP_n - 1$. If the exponential polynomial 
\[ p(t) = a_1 \lambda_1^t  + ... +  a_n \lambda_n^t \]
corresponding to the $ij$-th entry has non-vanishing derivative at $t = MIP_n - 1$, that is,
$p'(MIP_n - 1) \not = 0$
then, either 
$p(k)< 0$ for some $k>MIP_n - 1$, or 
$p(k)<0$ for all $ k \in (MIP_n - 1 -\epsilon, MIP_n - 1)$ with some $\epsilon$ small enough. In either case, the GDN critical exponent is at least $MIP_n-1$.
The only case in which the GDN critical exponent is less than $MIP_n-1$ is when the exponential polynomial corresponding to a certain entry has a multiple root at an integer larger than the critical exponent. Since the index of primitivity depends only on the number and positions of zeros in the matrix but not on the numerical values of nonzero entries, if it so happens that  $MIP_n-1 < CE_n$ with a certain matrix $A$, then the exponential polynomial corresponding to the entry where $A^{MIP_n-1}$ is zero for all GDN matrices with the same zero-nonzero pattern as $A$ has a multiple root at $MIP_n - 1$, which appears highly unlikely. Therefore, we make the following conjecture.

\begin{conj}\label{MIP<CE}
$MIP_n - 1 < CE_n$.
\end{conj}

If Conjecture \ref{MIP<CE} is true, then $MIP_n - 1$ gives a lower bound for the critical exponent, and the following question arises naturally.

\begin{question}\label{MIP_pattern}
What is $MIP_n$ and what are the zero-nonzero patterns that attain $MIP_n$?
\end{question}

Question \ref{MIP_pattern} not only may help improve the lower bound for GDN critical exponent but is also interesting in its own right. Note that as shown in Theorem 3.2 of \cite{Johnson}, the $MIP_n \ge n -1$ because of the tridiagonal DN matrices. By Lemma 2.4 in \cite{Johnson}, the maximum index of primitivity for DN matrices is precisely $n-1$. The next two lemmas gives an upper bound for $MIP_n$ and shows that $MIP_n > n-1$  if $n$ is odd.

\begin{lemma}\label{2n-d-1}
$MIP_n \le 2n - 3$.
\end{lemma}

\begin{proof}
Let $A$ be an $n$-by-$n$ matrix and first assume it has strict positive eigenvalues. Let $t_k = \text{Tr}(A^k)$. Then the characteristic polynomial of $A$ is given by:
\[ p(\lambda) = (-1)^n\big(\lambda^n + c_1\lambda^{n-1}+c_2\lambda^{n-2}+...+c_{n-1}\lambda + c_n\big)\]
where $c_1 = -t_1$ and $c_2 = \frac{1}{2}(t_1^2- t_2)$. 

By Descartes' rule of signs, the number of positive roots of $p(\lambda)$ is at most the number of the sign changes in the sequence $(1, c_1, ..., c_n)$. Hence, if $A$ is GDN, then $c_1<0$ and $c_2 > 0$. Therefore, $A$ has at least two positive diagonal entries. If $A$ has only one positive diagonal entry, then
\[2c_2 = \Big (\sum_{i=1}^n a_{ii} \Big ) - \sum_{i,j=1}^n a_{ij}a_{ji} = - \sum_{i\not = j}^n a_{ij}a_{ji} \le 0\]
which is impossble as $c_2>0$. For an irreducible matrix $A$ with at least one positive diagonal entry, it is a routine exercise to verify that the index of primitivity of $A$ is at most $2n - d - 1$, where $d$ is the number of positive diagonal entries (see e.g., Theorem 8.5.9 in \cite{MA}). Therefore, $MIP_n \le 2n - 3$. Finally, we relax the assumption that $A$ has positive eigenvalues as the general case follows from continuity.
\end{proof}

\begin{proposition}\label{n-1}
If $n$ is odd, then $MIP_n > n-1$ and $CE_n \ge n-1$. 
\end{proposition}

\begin{proof}

Consider the matrix 
\begin{equation}
A = \begin{bmatrix}
d_1  & \epsilon  & 0 & \cdots & \cdots & \cdots & \cdots & 0 \\
0  & d_2  & \epsilon  & \ddots & && & \vdots \\
0 & 0  & d_3 & \ddots  & \ddots & &  & \vdots \\
\vdots & \ddots & \ddots & \ddots & \ddots & \ddots &  & \vdots \\
\vdots & & \ddots & \ddots & \ddots & \ddots & \ddots& \vdots\\
\vdots & &  & \ddots & \ddots & d_{n-2} & \epsilon& 0\\
0  &  & & & \ddots & 0  & d_{n-1}  &  \epsilon\\
\epsilon & 0 &  \cdots & \cdots & \cdots & 0 &  0& d_n  \\
\end{bmatrix}
\end{equation}
where $d_1> d_2 > ... > d_{n-1} > d_n  = 0$ and $0< \epsilon< \max\{\frac{d_i - d_{i+1}}{2}\}$. By the Gershgorin circle theorem, all eigenvalues are real and the first $n-1$ eigenvalues are positive.
Moreover, since 
\[ \det(A) = \epsilon^n > 0 \]
all eigenvalues are positive.    
Therefore, $A$ is GDN. 

Note that $A^k_{nn}  = 0$ for $ k = 1,2, ..., n-1$. Hence the index of primitivity of $A$ is at least $n$. Since there are at most $n-1$ roots for the exponential polynomial $p(\alpha) = A^\alpha_{nn}$ and $A^{n}_{nn}>0$,  if follows that $A^\alpha_{nn}<0$ if $\alpha \in(n-2,n-1)$. Therefore, the critical exponent is at least $n-1$. 
\end{proof}

\begin{corollary}
$CE_3 = 2$.
\end{corollary}
\begin{proof}
The upper bound for the critical exponent is 2 by Theorem \ref{upper}, and the lower bound is also 2 by Proposition \ref{n-1}. 
\end{proof}

\begin{remark}
In Lemma 2.2 of \cite{Johnson}, it was shown that the diagonal entries of DN matrices remain positive under continuous powering, while in the proof of Proposition \ref{n-1}, a negative entry appears on the diagonal under continuous powering.
\end{remark}

It should be noted that when $n>3$, better lower bounds for $CE_n$ than $n-1$ exist, as demonstrated in the following examples. Examples of GDN matrices with highest index of primitivity discovered are displayed at the end of the section. 
\begin{example}

\begin{equation*}
A_4 = \begin{bmatrix}
           1     &      7   &        0     &      0\\
           0     &   17000   &      8500     &      0\\
           0     &      0   &    24000      &   1600\\
          20     &      0   &        0     &      5\\
\end{bmatrix}
\end{equation*}
The critical exponent for $A_4$ is between 3.99 and 4.
\begin{equation*}
A_5 = \begin{bmatrix}
10     &     70    &       0    &       0     &      0  \\
           0          & 5      &    90      &     0        &   0\\
           0         &  0      &  80000      &  15000         &  0\\
           0        &   0      &     0     &  120000         &  30\\
         150       &    0      &     0    &       0        &   0\\
\end{bmatrix}
\end{equation*}
The critical exponent for $A_5$ is between 5.99 and 6.

\begin{equation*}
A_6 = \begin{bmatrix}
         156    &    1605    &       0   &        0    &       0     &      0\\
           0    &     375    &    7932   &        0    &       0     &      0\\
           0    &       0    &     805   &     7840    &       0     &      0\\
           0    &       0    &       0   &  13803330    &   224210     &      0\\
           0    &       0    &       0   &        0    &  9373900     &   18590\\
       105720    &       0    &       0   &        0    &       0     &   25200\\
\end{bmatrix}
\end{equation*}
The critical exponent for $A_6$ is between 6.99 and 7.
\end{example}

\begin{remark}
Note that in the 4-by-4 case, the upper bound for critical exponent is 4 by Theorem \ref{upper} and the matrix $A_4$ in the previous example has critical exponent greater than 3.99. We notice that as row 1 and 4 of $A_4$ decrease in proportion (or equivalently as row 2 and 3 increase in proportion), the critical exponent increases. Therefore, we make the following conjecture.
\end{remark}

\begin{conj}\label{4}
$CE_4 = 4$.
\end{conj}

In Lemma \ref{2n-d-1}, we have shown that there have to be at least 2 positive entries on the diagonal. If there are exactly 2 positive entries on the diagonal, then the maximum index of primitivity is $2n-3$, giving $CE_n \ge 2n-4$ if Conjecture \ref{MIP<CE} holds. However, generally the zero-nonzero pattern with exactly two positive diagonal entries do not permit GDN matrices. Hence we perturb the diagonal zero entries and aim to the achieve lower bounds for $CE_n$ that are close to $2n-4$.  And we observe that when $n=4$ and $n=5$, we can perform such perturbation and produce CE greater than 3.99 and 5.99 respectively. Therefore, we ask the following question:
\begin{question}
Is $CE_n$ = $2(n-2)$?
\end{question}

\begin{table}\label{t}
\begin{center}

\begin{tabular}{ l | c | r|c }
    \hline
    n & $MIP_n$ & $CE_n$  & Upper bound for $CE_n$ (by Theorem 3.3)\\ \hline \hline
    2 & 1 & 0 & 0\\ \hline
    3 & 3 & 2 & 2\\ \hline
    4 & 4 & $>$3.99 &4 \\ \hline
    5 & 6 & $>$5.99 &7\\ \hline
    6 & 6 & $>$6.99 &12\\ \hline
    7 & 7 & $>$8.99 &16\\ \hline
       \hline
  \end{tabular}
\caption{Largest $CE_n$ and $MIP_n$ discovered in numerical experiments for small $n$'s}\label{table:t}
\end{center}

\end{table}

Table \ref{table:t} 
shows the highest $MIP_n$ and GDN CE discovered in numerical experiments.

Notice that for all these low dimension cases with $n>2$, the lower bounds for the critical exponent are strictly larger than $n-2$, the critical exponent for DN matrices.

Now we give examples of 4-by-4 GDN matrices with index of primitivity 4, 5-by-5 GDN matrices with index of primitivity 6, and 6-by-6 GDN matrices with index of primitivity 6.

\begin{example}

\begin{equation*}
A_4 = \begin{bmatrix}
           0     &      0   &        2     &      0\\
           0     &     68   &      56      &     21\\
           0     &      0   &       0      &    16\\
          14     &     72   &       0     &      168\\
\end{bmatrix}
\end{equation*}
The index of primitivity of $A_4$ is 4 and the GDN CE of $A_4$ at least 2.99.

\begin{equation*}
A_5 = \begin{bmatrix}
1800     &     405    &       0    &       0     &      0  \\
           0          & 916      &    794      &     0        &   0\\
           447         &  0      &   0      &  7        &  0\\
           0        &   300      &     0     &    0         &  15\\
         0       &    0      &     72    &       0        &   0\\
\end{bmatrix}
\end{equation*}
The index of primitivity of $A_5$ is 6 and the GDN CE of $A_5$ is at least 4.99.

\begin{equation*}
A_6 = \begin{bmatrix}
         2439   &    1020    &       0   &        0    &       0     &      0\\
           0    &     1917   &      668  &        0    &       0     &      0\\
           509  &       0    &      890  &      213    &       0     &      0\\
           0    &       2746 &       0   &        0    &     158     &      0\\
           0    &       0    &      270  &        0    &       0     &     2\\
           0    &       0    &       0   &      206    &       0     &     0\\
\end{bmatrix}
\end{equation*}
The index of primitivity of $A_6$ is 6 and the GDN CE of $A_6$ is at least 4.99.
\end{example}

\section{Additional observations}
We make a few observations about the reducible GDN matrices and about the Hadamard powering critical exponent of GDN matrices in this section.

\begin{proposition}
Let $A$ be a reducible $n$-by-$n$ GDN matrix. If 
\[ P^TAP = \begin{bmatrix}
B & C\\
0 & D \\
\end{bmatrix}
\] for some permutation matrix $P$, then 
\[P^TA^\alpha P = \begin{bmatrix}
B' & C'\\
0 & D' \\
\end{bmatrix}
\]
for all $\alpha > 0$. The matrix $B$ is $k$-by-$k$, $C$ is $k$-by-$n-k$, $0$ is $n-k$-by-$k$, and $D$ is $n-k$-by-$n-k$ for some integer $1<k<n$.
\end{proposition}
\begin{proof}
Let $p(\alpha)$ be the exponential polynomial corresponding to the $ij$-th entry, where $k+1 \le i \le n$ and $1 \le j \le k$, then $p(\alpha)$ has a root at every positive integer because the $ij$-th entry is zero for all integer powers of $A$. But $p(\alpha)$ has at most $n-1$ roots counting multiplicity if it is not identically zero, so
$ p(\alpha) \equiv 0$
Therefore, the $ij$-th entry stays zero under all continuous powers of $A$ and 
\[P^TA^\alpha P = \begin{bmatrix}
B' & C'\\
0 & D' \\
\end{bmatrix}
\]
where $B'$ is $k$-by-$k$, $C'$ is $k$-by-$n-k$, $0$ is $n-k$-by-$k$, and $D'$ is $n-k$-by-$n-k$, for all $t > 0$.
\end{proof}

Continuous powers of a GDN matrix
$A = (a_{ij})$ are also well defined under Hadamard multiplication. Namely, for $\alpha>0$ 
\[ A^{(\alpha)} = (a_{ij}^\alpha). \]
Contrary to the conventional multiplication, in the Hadamard case, entrywise nonnegativity is clear, but the nonnegativity of the eigenvalues is in question. It was shown in \cite{Hadamard} that the critical exponent for continuous Hadamard powering of doubly nonnegative matrices is also $n-2$. So it is natural to ask whether there exists a critical exponent without the symmetry condition and consider generalized doubly nonnegative matrices; however, the $n-2$ critical exponent does not generalize. In fact the critical exponent does not exist as demonstrated in the case below:

\begin{example}
If
\[A = 
\begin{bmatrix} 
  2    & 1 & 1 \\ 
  1    & 1 & 1  \\ 
  1    & 5 & 2   \\ 
\end{bmatrix},
\]
then the eigenvalues of $A^{(\alpha)}$ are $\lambda_1(\alpha) = 2^\alpha - 1$, $\lambda_2(\alpha) = 2^\alpha + \frac{1}{2} + \sqrt{5^\alpha + \frac{5}{4}}$ and $\lambda_3(\alpha) = 2^\alpha + \frac{1}{2} - \sqrt{5^\alpha + \frac{5}{4}}$.
Because $\lambda_3(\alpha) < 0$ for all $\alpha > 1$, the critical exponent does not exist.
\end{example}

\section{Questions}

In this section, we collect some questions that naturally arise when studying the GDN critical exponent. They are not only important and helpful in finding the GDN critical exponent, but are also interesting in their own right.

\begin{question}
Are GDN critical exponents for all $n$-by-$n$ matrices integers?
\end{question}
The critical exponent for both conventional and Hadamard powering of DN matrices turn out to be the integer $n-2$. It is natural to ask whether the same holds true in the conventional powering of GDN matrices. If that is indeed the case, then we can conclude $CE_4 = 4$ by the argument from section 4.
Moreover, in the case of conventional powering of DN, the maximum critical exponent is achieved by tridiagonal matrices. If $CE_n$ is also an integer and is achieved by a certain class of matrices, then we would have $2n-4$ as an upper bound for the critical exponent. To see that, suppose $A$ is a GDN matrix with the integer critical exponent $CE_n$. Then $A^{CE_n}$ has a zero entry. Because the index of primitivity of GDN matrices is at most $2n-3$ as shown in Lemma \ref{2n-d-1}, $CE_n \le 2n-4$.

\begin{question}
For which zero-nonzero patterns of primitive matrices do there exist GDN matrices?
\end{question}

\begin{question}
What is the relation between $MIP_n$ and $CE_n$?
\end{question}
We have seen in section 4 that $MIP_n$ is closely related to $CE_n$, and the knowledge of the relation between $MIP_n$ and $CE_n$ would help us gain information on one given knowledge about the other.

\bibliographystyle{gLMA}
\bibliography{mybibfile}

\end{document}